\documentclass[11pt,a4paper]{article}
\usepackage{mathrsfs}
\usepackage{amsmath,amsthm,color,eepic}
\usepackage{amssymb} \usepackage{verbatim}
\usepackage{lineno}

\theoremstyle{theorem}
\newtheorem{thm}{Theorem}[section]

\newtheorem{lemma}{Lemma}[section]

\newtheorem{fact}{Fact}

\theoremstyle{definition}

\newtheorem{claim}{Claim}

\newtheorem{conj}{Conjecture}

\begin{document}
\title{\bf  A Strengthening of Erd\H{o}s-Gallai Theorem
and Proof of Woodall's Conjecture}
\date{}

\author{Binlong Li\footnote{School of Mathematics and Statistics, Northwestern Polytechnical University, Xi'an, 710129, China, and Xi'an-Budapest Joint Research Center for Combinatorics, Northwestern Polytechnical University, Xi'an, Shaanxi 710129, China. Partially supported by NSFC (No. 11601429) and the Fundamental Research Funds for the Central Universities (3102019ghjd003).}~~~Bo Ning\footnote{School of Mathematics, Tianjin University, Tianjin, 300072, P.R. China.}
\footnote{Corresponding author. The current address: College of Computer Science,
Nankai University, Tianjin, 300071, P.R. China.
E-mail: bo.ning@nankai.edu.cn (B. Ning). Partially supported by NSFC (No. 11971346).}} \maketitle

\begin{center}
\begin{minipage}{140mm}
\small\noindent{\bf Abstract:} For a 2-connected graph $G$
on $n$ vertices and two vertices $x,y\in V(G)$, we prove that
there is an $(x,y)$-path of length at least $k$, if there are
at least $\frac{n-1}{2}$ vertices in $V(G)\backslash \{x,y\}$
of degree at least $k$. This strengthens a celebrated theorem due
to Erd\H{o}s and Gallai in 1959. As the first application of
this result, we show that a 2-connected graph with $n$
vertices contains a cycle of length at least $2k$, if it
has at least $\frac{n}{2}+k$ vertices of degree at least
$k$. This confirms a 1975 conjecture made by Woodall.
As other applications, we obtain some results which
generalize previous theorems of Dirac, Erd\H{o}s-Gallai,
Bondy, and Fujisawa et al., present short proofs of the
path case of Loebl-Koml\'{o}s-S\'{o}s Conjecture which
was verified by Bazgan et al. and a conjecture of
Bondy on longest cycles (for large graphs) which was
confirmed by Fraisse and Fournier, and make progress on
a conjecture of Bermond.

\end{minipage}
\end{center}
\section{Introduction}
For a graph $G$ and $x,y \in V(G)$, an $(x,y)$-path of
$G$ is a path with two end-vertices $x$ and $y$.
The \emph{length} of a path is the number of edges in it. The
study on the longest $(x,y)$-paths has a long history.
The famous Erd\H{o}s-Gallai Theorem \cite{EG59} asserts
that for any positive integer $k$ and two distinct vertices $x,y$ in
a 2-connected graph $G$, if every vertex other than $x,y$ has
degree at least $k$, then there is an $(x,y)$-path of length
at least $k$ in $G$. For a graph $G$ and $v\in V(G)$, the
{\em neighborhood} $N_G(v)$ of $v$ in $G$ is the set of vertices which
are adjacent to $v$. The {\em degree} of $v$ in $G$, denoted
by $d_G(v)$, equals $|N_G(v)|$. For a graph $G$ and $x,y \in V(G)$,
let $n_k(x,y)=|\{z\in V(G)\setminus\{x,y\}: d_G(z) \geq k\}|$.
By Erd\H{o}s-Gallai Theorem, if $G$ is 2-connected
and $n_k(x,y)=n-2$, then $G$ contains an $(x,y)$-path of
length at least $k$. Bondy and Jackson \cite{BJ85} showed
that a weaker condition that $n_k(x,y) \geq n-3$, $n\geq 4$, suffices.
(A slightly weaker result can be found in Alon \cite[Lemma~2.3]{A86}.)
It is natural to ask what is the best lower bound on
$n_k(x,y)$ which still ensures  the same conclusion.
We use a novel method to prove the following.
\begin{thm}\label{Thm:strengthening-ErdosGallai}
Let $G$ be a 2-connected graph on $n$ vertices and
$x,y\in V(G)$. If there are at least $\frac{n-1}{2}$ vertices
in $V(G)\backslash \{x,y\}$ of degree at least $k$, then
$G$ contains an $(x,y)$-path of length at least $k$.
\end{thm}

We shall present an example which shows that the condition in
Theorem \ref{Thm:strengthening-ErdosGallai} is best
possible. In order to describe the example, we need
to introduce three operations on graphs here. Let $G_1$
and $G_2$ be two vertex-disjoint graphs. The
\emph{(disjoint) union} of $G_1$ and $G_2$, denoted
by $G_1+G_2$, is a new graph with vertex set
$V(G_1)\cup V(G_2)$ and edge set $E(G_1)\cup E(G_2)$.
Let $G_1\vee G_2$ denote $G_1+G_2$ together with all
edges from $V(G_1)$ to $V(G_2)$. We use $\overline{G}$
to denote the \emph{complement} of $G$.

Let $k\geq 5$ be an odd integer and $t$ be a positive
integer. We construct a graph $G$ as follows: starting
from $t$ disjoint copies of
$K_{\frac{k-1}{2}}\vee \overline{K_{\frac{k-1}{2}}}$,
we add two new vertices $x$ and $y$ followed by adding
all edges between $x,y$ and vertices in each copy of
$K_{\frac{k-1}{2}}$. Notice that there are
exactly $\frac{|G|-2}{2}$ vertices other than
$x,y$ of degree at least $k$,
and each longest $(x,y)$-path in $G$ is of length $k-1$.

As the first application of Theorem
\ref{Thm:strengthening-ErdosGallai}, we confirm the
following long-standing conjecture by Woodall \cite{W75},
which improves the famous Dirac's theorem \cite{D52}
in very strong sense. For a graph $G$,
the {\it circumference} $c(G)$
is the length of a longest cycle in $G$.

\begin{conj}[Woodall \cite{W75}]\label{Conj:Woodall}
Let $G$ be a 2-connected graph on $n$ vertices.
If there are at least $\frac{n}{2}+k$ vertices
of degree at least $k$, then $c(G)\geq 2k$.
\end{conj}

This conjecture was listed as one of 50 unsolved
problems in the textbook by Bondy and Murty
(see \cite[Problem~7,~Appendix~IV]{BM76}).
It has attracted wide attention since then. In 1985,
H\"{a}ggkvist and Jackson \cite{HJ85} showed the
conclusion in this conjecture holds if the graph
$G$ satisfies either of the following conditions:
(a) $G$ has at most $3k-2$ vertices and at least
$2k$ vertices of degree at least $k$; or (b) $G$
has $n\geq3k-2$ vertices and at least
$n-\frac{k-1}{2}$ vertices of degree at least $k$.
Li and Li \cite{LL} verified the conjecture for
the case where $n \leq 4k-6$. If $G$ is 3-connected,
H\"{a}ggkvist and Li (unpublished, see \cite{L02})
confirmed the conjecture for $k\geq 25$. If we do
not assume any further conditions, in 2002,
Li \cite{L02} showed $c(G)\geq 2k-13$ based
on the concept of $(k,B)$-connectivity and vines
of paths (see \cite{T66,B71,L82}). It was remarked
that ``in \cite{L}, our complete proof of the
conjecture for $k\geq 683$ was much longer"
\footnote{We quote this sentence from \cite{L02}.}.
For other related results, we refer interested
readers to a survey (see \cite[Section~4]{L13}).
To authors' best  knowledge, a complete proof of
Woodall's conjecture for all $k\geq 2$ is still open.

In this paper, we resolve Woodall's conjecture completely.

\begin{thm}\label{Thm:WoodallConj}
Let $G$ be a 2-connected graph on $n$ vertices.
If there are at least $\frac{n}{2}+k$ vertices
of degree at least $k$, then $c(G)\geq 2k$.
\end{thm}

Using Theorem \ref{Thm:WoodallConj} as a tool, we can obtain a
partial solution to a conjecture of Bermond \cite{B76} on circumference
of a 2-connected graph. We refer the reader to
Subsection \ref{Subsec:BermondConj} for details.

Besides a proof of Conjecture \ref{Conj:Woodall},
Theorem \ref{Thm:strengthening-ErdosGallai} has
other applications, including, for example,
a Woodall-type Fan Lemma, and two generalizations
of Erd\H{o}s-Gallai Theorems on paths and cycles
under an independent set condition, respectively.
Moreover, with  Theorem \ref{Thm:strengthening-ErdosGallai}
in hand, we are able to present short proofs of Bondy's
conjecture \cite{B81} for large graphs and Loebl-Koml\'{o}s-S\'{o}s
Conjecture for paths. We should point
out that Fournier and Fraisse \cite{FF85} verified
Bondy's conjecture and Bazgan et al. \cite{BLW00} proved
Loebl-Koml\'{o}s-S\'{o}s Conjecture for paths. We will discuss
details in Subsection \ref{Subsec:Another-App}.

Throughout this paper, all graphs are simple and finite.
For a graph $G$ and $v \in V(G)$, the {\it closed neighborhood}
$N_G[v]$ is the set $N_G(v)\cup \{v\}$. If $S \subset V(G)$,
then $N_S(v):=N_G(v)\cap S$ and $d_S(v)=|N_S(v)|$. Moreover,
we use $G[S]$ to denote the subgraph of $G$ induced by $S$
and $G-S$ the subgraph of $G$ induced by $V(G)\backslash S$.
If $S=\{v\}$, we write $G-v$ instead of $G-\{v\}$. For a
subgraph $H$ of $G$, we define $N_S(H)=\bigcup_{v\in V(H)}N_S(v)$.
When there is no danger of ambiguity, for subgraphs $H'$ and
$H$ of $G$, We use $N_H(v)$, $d_H(v)$ and $N_H(H')$ instead of
$N_{V(H)}(v)$, $d_{V(H)}(v)$ and $N_{V(H)}(H')$, respectively.

For a path $P$ and $u,v\in V(P)$, let $P[u,v]$ be the segment of $P$
from $u$ to $v$. For a path $P$ with terminus $x$ and
a path $Q$ with origin $x$, if the union $P\cup Q$ is again
a path, we may simply denote it by $PxQ$.
For a separable graph $G$, a \emph{block} is a maximal non-separable
subgraph of $G$, and an \emph{end-block} is a block which contains
exactly one cut-vertex of $G$. For
an end-block $B$ and a cut-vertex $v\in V(B)$,
every vertex in $V(B)\backslash \{v\}$ is
called an \emph{inner-vertex}
of $B$. For $s\leq t$, let $[s,t]$ be
the set of integers $i$ with $s\leq i\leq t$.
For those notation not defined here, we
refer the reader to \cite{BM76}.

The rest of this paper is organized as follows.
All applications of Theorem \ref{Thm:strengthening-ErdosGallai}
are included in Section \ref{Sec:App}. The section
has three parts. In Subsection \ref{Subsec:WoodallConj},
we shall present a proof of Woodall's conjecture (assuming
Theorem \ref{Thm:strengthening-ErdosGallai}).
In Subsection \ref{Subsec:BermondConj}, we shall give a partial
solution to a conjecture of Bermond on circumference of
graphs with the help of Woodall's conjecture.
In Subsection \ref{Subsec:Another-App}, we shall present
several other applications of Theorems
\ref{Thm:strengthening-ErdosGallai} and \ref{Thm:WoodallConj}.
The proof of Theorem \ref{Thm:strengthening-ErdosGallai}
will be postponed to Section \ref{Sec:MainThm}.
In the last section, we will remark that a construction from
H\"{a}ggkvist and Jackson can disprove a conjecture
of Li \cite[Conjecture~4.14]{L13}
and mention a conjecture generalizing
Theorem \ref{Thm:strengthening-ErdosGallai}.

\section{Applications}\label{Sec:App}
\subsection{Proof of Woodall's conjeccture}\label{Subsec:WoodallConj}
The goal of this subsection is to prove Woodall's conjecture (assuming Theorem
\ref{Thm:strengthening-ErdosGallai}).

Let $G$ be a graph, $C$ a cycle of $G$, and $H$ a component
of $G-C$. A subgraph $F$ is called an \emph{$(H,C)$-fan},
if it consists of paths $P_1,P_2,\ldots,P_t$ where $t\geq 2$,
such that: (1) all $P_i$ have the same origin $v\in V(H)$
and pairwise different termini $u_i\in V(C)$, $1\leq i\leq t$;
(2) all internal vertices of $P_i$ are in $H$ and $P_i$'s are
pairwise internally disjoint.

We shall first prove a Woodall-type Fan Lemma.

\begin{thm}\label{Thm:Woodall-FanLemma}
Let $G$ be a 2-connected graph, $C$ a cycle of $G$, and $H$ a
component of $G-C$. If there are at least $\frac{|H|+1}{2}$
vertices in $V(H)$ of degree at least $k$ in $G$,
then there is an $(H,C)$-fan with at least $k$ edges.
\end{thm}

\begin{proof}
Set $t:=\max \{d_C(u):u\in V(H)\}$. First suppose $t=1$. Choose
$x_1\in N_C(H)$. Since $G$ is 2-connected, $d_C(H)\geq 2$. If
$N_H(x_1)=N_H(C)$, we have $t\geq 2$, a contradiction. Thus
$N_H(C)\neq N_H(x_1)$. We construct a new graph $H'$ from $H$
by adding two new vertices $x,y$ with edge set
$$E(H')=E(H)\cup\{xv:v\in N_H(x_1)\}\cup
\{yv:v\in N_H(C)\backslash N_H(x_1)\}\cup\{xy\}.$$
Obviously, $H'$ is 2-connected.
Furthermore, $d_{H'}(u)=d_G(u)$ for any vertex $u\in V(H)$.
By Theorem \ref{Thm:strengthening-ErdosGallai}, $H'$
has an $(x,y)$-path of length at least $k$. Let $P:=xv_1\cdots v_py$
be such a path. Since $v_1\in N_{H}(x_1)$ and $v_p\in
N_H(C)\backslash N_H(x_1)$, there exists $y_1\in N_C(v_p)$ such that
$y_1\neq x_1$. Thus, $P'=x_1v_1\ldots v_py_1$ is an $(x_1,y_1)$-path
of length at least $k$ with all internal vertices in $H$. Such an
$(x_1,y_1)$-path is an $(H,C)$-fan we seek.

Now suppose $t\geq 2$. Let $U_1:=\{u\in V(H):d_C(u)=t\}$ and
$U_2:=\{u\in V(H):d_C(u)\geq 1\}$. Construct a new graph $H'$ from $H$
by adding two vertices $x,y$ with edge set
$$
E(H')=E(H)\cup\{xu:u\in U_1\}\cup \{yu:u\in U_2\}\cup \{xy\}.
$$
Then $H'$ is 2-connected. If $t\geq k$, then there is
an $(H,C)$-fan of length at least $k$ which is a star. So, assume
that $t\leq k-1$. Notice that every vertex in $H$ of degree
at least $k$ in $G$ has degree at least $k-t+2$ in $H'$. By Theorem
\ref{Thm:strengthening-ErdosGallai}, $H'$ has an $(x,y)$-path, say
$P=xv_1\cdots v_{p'}y$, of length at least $k-t+2$. Observe that $v_{p'}$ has at
least one neighbor on $C$, and $v_1$ has at least $t\geq 2$
neighbors on $C$. Let $y_1\in N_C(v_{p'})$ and $x_1\in
N_C(v_1)\backslash\{y_1\}$. Thus, $P':=x_1v_1\ldots v_{p'}y_1$ is an
$(x_1,y_1)$-path of $G$ with all internal vertices in $H$. The path $P'$,
together with all edges $zv_1$ with $z\in N_C(v_1)\backslash \{x_1,y_1\}$,
shall create an $(H,C)$-fan with at least $k$ edges. The proof of
Theorem \ref{Thm:Woodall-FanLemma} is complete.
\end{proof}

Now we are ready to prove Woodall's conjecture, which needs
Theorem \ref{Thm:Woodall-FanLemma} and a well-known fact
as following.
\begin{lemma}\label{Lem:wellknown}
Let $G$ be a 2-connected nonhamiltonian graph, $C$ a longest cycle and $H$
a component of $G-C$. If there is an $(H,C)$-fan with at least $k$
edges, then $c(G)\geq 2k$.
\end{lemma}

\noindent
{\bf Proof of Theorem \ref{Thm:WoodallConj}.}
Let $C$ be a longest cycle of $G$. By the condition $\frac{n}{2}+k\leq n$,
we infer $n\geq 2k$. If $C$
is a Hamilton cycle, then we have $|C|\geq 2k$. Thus, $G$
is not Hamiltonian. Let $\mathcal{H}=\{H_1,H_2,\ldots,H_t\}$ be
a collection of all components of $G-C$. If there exists an integer $i\in[1,t]$,
such that $H_i$ contains at least $\frac{|H_i|+1}{2}$ vertices
of degree at least $k$ in $G$, then by Theorem \ref{Thm:Woodall-FanLemma},
there is an $(H_i,C)$-fan with at least $k$ edges. By Lemma \ref{Lem:wellknown}
we have $|C|\geq 2k$. Therefore, $H_i$ contains at most $\frac{|H_i|}{2}$
vertices of degree at least $k$ for each $i\in [1,t]$. This implies that the number
of vertices of degree at least $k$ in $G$ is at most
$$|C|+\sum_{i=1}^t\frac{|H_i|}{2}=\frac{n+|C|}{2}.$$
Hence $\frac{n+|C|}{2}\geq\frac{n}{2}+k$ and this implies
$|C|\geq 2k$, completing the proof of Theorem \ref{Thm:WoodallConj}. {\hfill$\Box$}

\subsection{On a conjecture of Bermond}\label{Subsec:BermondConj}
Generalizing the classical degree conditions for Hamilton
cycles, Bermond \cite{B76} proposed the following conjecture.
The conjecture was recalled in Dean and Fraisse \cite{DF89},
and also listed as a conjecture in a monograph of
Bollob\'as (see \cite[Conjecture~32,~pp.~296]{B78}).
\begin{conj}[Bermond \cite{B76}]\label{Conj-Bermond-1}
Let $G$ be a 2-connected graph with vertex set $V=\{x_i:1\leq i\leq n\}$ and $c$ be a positive
integer, where $c\leq n$. If for every pair of vertices
$x_i,x_j$, $i<j$, one of the following holds:
$$
(\mbox{i})~i+j<c;\ (\mbox{ii})~x_ix_j\in E(G);\ (\mbox{iii})~d(x_i)>i;\  (\mbox{iv})~d(x_j)\geq j;\  (\mbox{v})~d(x_i)+d(x_j)\geq c,
$$
\noindent
then $c(G)\geq c$.
\end{conj}

With the help of Theorem \ref{Thm:WoodallConj}, we can obtain a
partial solution to Conjecture \ref{Conj-Bermond-1}.
\begin{thm}
Conjecture \ref{Conj-Bermond-1} is true if $n\geq 3c-1$.
\end{thm}

\begin{proof}
In this proof, we say that a vertex is \emph{feasible}
if its degree is at least $\frac{c}{2}$.

We show that under the condition of Conjecture \ref{Conj-Bermond-1},
$G$ has at most $c-1$ non-feasible vertices. If this is
already proved, then $G$ will have at least
$\frac{n}{2}+\lceil\frac{c}{2}\rceil$ feasible vertices
(when $n\geq3c-1$) and so $c(G)\geq c$ by
Theorem \ref{Thm:WoodallConj}.

We first show that there exists a non-feasible vertex,
say $x_i$, such that $d(x_i)\leq i$. Indeed, if not,
then every non-feasible vertex $x_k$ satisfies that $k<d(x_k)<\frac{c}{2}$.
Thus, there are at most $\lceil\frac{c}{2}\rceil-2$ non-feasible vertices,
and we are done.

Now we choose a non-feasible vertex $x_i$ with $d(x_i)\leq i$ and $i$ is
as small as possible.
Suppose first that $i<\frac{c}{2}$. In this case, every
non-feasible vertex $x_j$ satisfies that either $j<c-i$ or $x_ix_j\in E(G)$.
In fact, if $j\geq c-i$ and $x_ix_j\notin E(G)$, then the pair $(x_i,x_j)$
satisfies none of the conditions (i)-(v). It follows that there are
at most $|\{x_j: j<c-i\}|+|N(x_i)|\leq c-i-1+i=c-1$ non-feasible vertices.
Secondly, suppose that $i\geq \frac{c}{2}$. Let $x_j$ be a non-feasible
vertex other that $x_i$. If $j<i$, then by the choice of $x_i$, we have
$j<d(x_j)<\frac{c}{2}$, i.e., $j\leq \lceil\frac{c}{2}\rceil-2$.
If $j>i$, then $x_ix_j\in E(G)$ (for otherwise the pair $(x_i,x_j)$
satisfies none of the conditions (i)-(v)). It follows that there
are at most
$$\left|\left\{x_j: j\leq \left\lceil \frac{c}{2}\right\rceil-2\right\}\right|+|\{x_i\}|+|N(x_i)|\leq \left(\left\lceil\frac{c}{2}\right\rceil-2\right)+1+\left(\left\lceil\frac{c}{2}\right\rceil-1\right)\leq c-1$$
non-feasible vertices. The proof is complete.
\end{proof}
\subsection{Other applications of Theorems \ref{Thm:strengthening-ErdosGallai}
and \ref{Thm:WoodallConj}}\label{Subsec:Another-App}
In this subsection, we present several other consequences of
Theorem \ref{Thm:strengthening-ErdosGallai}.
A famous consequence of Menger's theorem is known as Fan Lemma
as follows.
\begin{thm}(Dirac \cite{D52})
Let $G$ be a $k$-connected graph, $v\in V(G)$ and $Y\subseteq V(G)\backslash \{v\}$
with $|Y|\geq k$. Then there are $k$ internally disjoint paths from $v$ to $Y$
whose termini are distinct.
\end{thm}

The following is a variant of Fan Lemma for 2-connected graphs,
which is a corollary of Theorem \ref{Thm:Woodall-FanLemma}.
\begin{thm}\label{Thm:FanLemma}
Let $G$ be a 2-connected graph, $C$ a cycle of $G$, and $H$ a
component of $G-C$. If each vertex in $H$ has degree at least $k$
in $G$, then there is an $(H,C)$-fan with at least $k$ edges.
\end{thm}
We point out that Fujisawa et al. \cite{FYZ05} proved a stronger theorem
than Theorem \ref{Thm:FanLemma}.

Bazgan, Li and Wo\'{z}niak \cite{BLW00} confirmed the
famous Loebl-Koml\'{o}s-S\'{o}s conjecture
for paths. Their result is a direct corollary of our
Theorem \ref{Thm:strengthening-ErdosGallai}.
\begin{thm}[Bazgan, Li, Wo\'{z}niak \cite{BLW00}]\label{Thm:BLW}
Let $G$ be a graph on $n$ vertices. If there are at least $\frac{n}{2}$ vertices
of degree at least $k$, then $G$ contains a path of length at least $k$.
\end{thm}
\begin{proof}
We only need to prove the theorem for the case of $G$ being
connected. Let $G'$ be obtained from $G$ by adding
a new vertex $x$ and joining $x$ to all vertices
in $G$. Then $G'$ is 2-connected. If every vertex
in $G$ has degree at least $k$, then choose $y\in V(G)$
arbitrarily; otherwise, choose $y\in V(G)$ such
that $d_G(y)\leq k-1$. There are at least
$\frac{|G'|-1}{2}$ vertices in $V(G')\backslash \{x,y\}$ of degree at
least $k+1$. By Theorem \ref{Thm:strengthening-ErdosGallai},
there is an $(x,y)$-path $P$ of length at least
$k+1$. Deleting the vertex $x$ in $P$ gives us a
required path. This proves Theorem \ref{Thm:BLW}.
\end{proof}

Theorem \ref{Thm:strengthening-ErdosGallai} implies
a generalization of Erd\H{o}s-Gallai Theorem
under an independent set condition.
\begin{thm}
Let $k,s\geq 1$ and $G$ be a 2-connected graph
on $n\geq 2ks+3$ vertices and $x,y\in V(G)$.
If $\max\{d(v):v\in S\}\geq k$ for any independent
set $S\subset V(G)\backslash \{x,y\}$ with
$|S|=s+1$, then $G$ has an $(x,y)$-path of
length at least $k$.
\end{thm}
\begin{proof}
Let $I$ be a maximal independent set such that for every vertex $u\in I$,
$d(u)<k$ and $u\notin\{x,y\}$. By condition we have $|I|\leq s$.
For any vertex $v\notin\bigcup_{u\in I}N[u]\cup \{x,y\}$, we have
$d(v)\geq k$ by the choice of $I$. It follows that there are at least
$|G|-2-\sum_{u\in I}|N[u]|\geq n-2-ks\geq\frac{n-1}{2}$ vertices
in $V(G)\backslash \{x,y\}$ of degree at least $k$. The result follows from
Theorem \ref{Thm:strengthening-ErdosGallai}.
\end{proof}

Theorem \ref{Thm:WoodallConj} can imply two classical
theorems and a generalization
under an independent set condition.

\begin{thm}[Dirac \cite{D52}]\label{Thm:Dirac-52}
Let $G$ be a 2-connected graph on $n$ vertices. If the degree of every
vertex is at least $k$, then $c(G)\geq \min\{n,2k\}$.
\end{thm}

\begin{thm}[\rm {Erd\H{o}s, Gallai \cite[pp.~344]{EG59}, Bondy \cite{B71}}]\label{Thm:Bondy}
Let $G$ be a 2-connected graph on $n$ vertices. If the degree of every
vertex other than one vertex is at least $k$, then $c(G)\geq \min\{n,2k\}$.
\end{thm}

\begin{thm}\label{Thm:sigmak+1}
Let $k,s\geq 1$ and $G$ be a $2$-connected graph on
$n\geq 2k(s+1)$ vertices. If $\max\{d(v):v\in S\}\geq k$
for any independent set $S\subset V(G)$ with $|S|=s+1$,
then $c(G)\geq 2k$.
\end{thm}
\begin{proof}
Let $I$ be a maximal independent set such that for every vertex $u\in I$, $d(u)<k$.
By condition we have $|I|\leq s$. For any vertex $v\notin\bigcup_{u\in I}N[u]$, we have $d(v)\geq k$ by the
choice of $I$. It follows that there are at least
$|G|-\sum_{u\in I}|N[u]|\geq n-ks\geq\frac{n}{2}+k$ vertices of degree at least $k$.
Theorem \ref{Thm:sigmak+1} follows from Theorem \ref{Thm:WoodallConj}.
\end{proof}

Theorem \ref{Thm:sigmak+1}
implies Fournier and Fraisse's theorem (for large graphs) which was originally
conjectured by Bondy \cite{B81}.

\begin{thm}[Fournier, Fraisse \cite{FF85}]\label{Thm:Bondyconj}
Let $G$ be an $s$-connected graph on $n$ vertices where $s\geq 2$. If the
degree sum of any independent set of size $s+1$ is at least $m$, then
$c(G)\geq\min\{\frac{2m}{s+1},n\}$.
\end{thm}

\section{Proof of Theorem \ref{Thm:strengthening-ErdosGallai}}\label{Sec:MainThm}
In this section, we prove Theorem \ref{Thm:strengthening-ErdosGallai}.
We first introduce the concept of \emph{Kelmans operation} and
prove a lemma.

Let $G$ be a graph and $u,v \in V(G)$. A new graph $G'$
is a \emph{Kelmans graph} of $G$ (from $u$ to $v$), denoted by
$G':=G[u\rightarrow v]$, if $V(G')=V(G)$ and
$E(G')=\left(E(G)\backslash \{uw:w\in N[u]\backslash N[v]\}\right)\cup \{vw:w\in N[u]\backslash N[v]\}.$
The operation was originally studied by Kelmans in \cite{K81}.
It is also called edge-switching and is a powerful tool for
solving problems on long cycles and cycle covers of graphs
(e.g. \cite{F02,MN20}).

Let $G$ and $G'$ be two graphs of order $n$.
Assume $\tau(G)=(d_1,d_2,\ldots,d_n)$ and $\tau(G')=(d'_1,d'_2,\ldots,d'_n)$
are non-increasing degree sequences of $G$ and $G'$, respectively.
If there exists an integer $j$ such that $d_k=d'_k$ for $1 \leq k \leq j-1$
and $d_j>d'_{j}$, then we say that $\tau(G)$ is \emph{larger than} $\tau(G')$
and denote it by $\tau(G)>\tau(G')$.

We have the following lemma.
\begin{lemma}\label{Lemma:Kelmans}\footnote{This lemma was also discovered by Hehui
Wu \cite{W}, independently.}
Let $G$ be a graph, $x,y,u$ be distinct vertices of $G$, and
$v\in N(u)$ (possibly $v\in\{x,y\}$). Let $G':=G[u\rightarrow v]$. \\
(i) If neither $N[u]\subseteq N[v]$ nor $N[v]\subseteq N[u]$, then
$\tau(G')>\tau(G)$.\\
(ii) If $G'$ has an $(x,y)$-path of length at least $k$, then so
does $G$.
\end{lemma}

\begin{proof}
(i) Note that $G[u\rightarrow v]$ is isomorphic to $G[v\rightarrow u]$.
By the assumption, we have $N(u)\backslash N(v)\neq \emptyset$ and $N(v)\backslash N(u)\neq \emptyset$.
Without loss of generality, we assume  $d_G(u)\leq d_G(v)$. Then
$d_{G'}(v)>d_G(v)\geq d_G(u)>d_{G'}(u)$ and the degrees of all other
vertices are not changed. It follows $\tau(G')>\tau(G)$.

(ii) Assume $P$ is an $(x,y)$-path of length at least $k$ in $G'$. If
$P$ does not pass through $v$, then it is also a path in $G$ and we
are done. Thus we only consider the case where $P$ passes
through $v$. Let $v^-$ and $v^+$ be the predecessor and successor
of $v$ on $P$. If $v=x$, then we do not define $v^-$; if $v=y$,
then we do not define $v^+$. For
$w \in N_{G'}(v)\setminus N_{G}(v)$, we note $w \in N_G(u)$
by the definition of $G'$. Suppose $P$
does not pass through $u$. If $v=x$, then
\[
P'=\left\{\begin{array}{ll}
P,                        & \mbox{if } vv^+\in E(G);\\
vuv^+P[v^+,y],             & \mbox{if } vv^+\notin E(G)
\end{array}\right.
\]
is an $(x,y)$-path in $G$ of length at least $k$. If $v=y$, then we can prove it similarly.
Thus, we can assume $v \notin \{x,y\}$. Set
\[
P'=\left\{\begin{array}{ll}
P,                        & \mbox{if } vv^-\in E(G) \mbox{ and } vv^+ \in E(G);\\
P[x,v]vuv^+P[v^+,y],      & \mbox{if }vv^-\in E(G)\mbox{ and } vv^+ \notin E(G);\\
P[x,v^-]v^-uvP[v,y],      & \mbox{if }vv^- \notin E(G)\mbox{ and } vv^+ \in E(G);\\
P[x,v^-]v^-uv^+P[v^+,y],  & \mbox{if }vv^- \notin E(G)\mbox{ and } vv^+ \notin E(G).
\end{array}\right.
\]
Clearly $P'$ is an $(x,y)$-path of $G$ of length at least
$k$.

We now consider the case where $P$ also
passes through $u$.  Let $u^-,u^+$ be the
predecessor and successor of $u$ on $P$, respectively.
Suppose that $x=v$.
Then
$$P'=\left\{\begin{array}{ll}
  P,                        & \mbox{if } vv^+\in E(G);\\
  xu^-P[u^-,v^+]v^+uP[u,y],      & \mbox{if }vv^+\notin E(G)
\end{array}\right.$$
is an $(x,y)$-path of length at least $k$ in $G$. The case
of $y=v$ can be solved similarly. Thus, $v\notin \{x,y\}$.
In the following, we assume $x,v,u,y$ appear on
$P$ in order. We define the path $P'$ as follows:
$$P'=\left\{\begin{array}{ll}
  P,                        & \mbox{if } vv^-\in E(G)\mbox{ and } vv^+\in E(G);\\
  P[x,v]vu^-P[u^-,v^+]v^+uP[u,y],      & \mbox{if }vv^-\in E(G)\mbox{ and }vv^+\notin E(G);\\
  P[x,v^-]v^-uP[u,v]vu^+P[u^+,y],      & \mbox{if }vv^-\notin E(G)\mbox{ and }vv^+\in E(G);\\
  P[x,v^-]v^-uv^+P[v^+,u^-]u^-vu^+P[u^+,y],  & \mbox{if }vv^-\notin E(G)\mbox{ and }vv^+\notin E(G).
\end{array}\right.$$
Then $P'$ is an $(x,y)$-path we seek. The proof is complete.
\end{proof}

The following lemma will be used frequently
in the proof of Theorem \ref{Thm:strengthening-ErdosGallai}.

\begin{lemma}\label{Lem:2connected2}
(i) If $G$ is 2-connected and $G-v$ is separable for
a vertex $v\in V(G)$, then every end-block of $G-v$ has
at least one inner-vertex that is adjacent to $v$ in $G$.

\noindent
(ii) Assume that $G$ is separable in which $v\in V(G)$ is not a cut-vertex.
Let $B_1,B_2,\ldots,B_t$ be all end-blocks of $G$ not containing $v$.
Let $G'$ be obtained from $G$ by adding
edges $vv_i$, where $v_i$ is an inner-vertex of $B_i$, $i\in [1,t]$.
Then $G'$ is 2-connected.

\noindent
(iii) Assume that $G$ is separable. Let $B_1,B_2,\ldots,B_t$ be all end-blocks of $G$.
Let $G'$ be obtained from $G$ by adding a new vertex $v$ and
new edges $vv_i$, where $v_i$ is an inner-vertex of $B_i$,
$i\in [1,t]$. Then $G'$ is 2-connected.

\noindent (iv) Assume that $G$ is 2-connected and $\{x,y\}$ be a cut
of $G$. Let $H_1,\ldots,H_t$, $t\geq 1$, be some (not necessary all)
components of $G-\{x,y\}$, and $S=\bigcup_{i=1}^t V(H_i)$. If $xy\in
E(G)$, then let $G'=G[S\cup\{x,y\}]$; if $xy\notin E(G)$, then let
$G'$ be obtained from $G[S\cup\{x,y\}]$ by adding the edge $xy$.
Then $G'$ is 2-connected.

\noindent (v) Let $G$ be 2-connected. Let $v\in V(G)$ such that
its neighborhood is a clique. If $G-v$ has order at least 3, then
$G-v$ is 2-connected.
\end{lemma}

\begin{proof}
(i) For an end-block $B$ of $G-v$ and a cut-vertex $u$ of $G-v$ contained in $B$,
if $N_{B-u}(v)=\emptyset$, then $u$ is a cut-vertex of $G$, a contradiction.

(ii) Choose $u\in V(G')$ arbitrarily. If $u$ is not a cut-vertex of $G$,
then $G-u$ is connected, and so $G'-u$ is connected. If $u$ is a
cut-vertex of $G$, then $u\neq v$. As $v$ is adjacent
to any end-block of $G$ not containing $v$ in $G'$, $v$ is adjacent to
any component of $G'-u$ not containing $v$. In this case, $G'-u$ is connected.
Thus, $G'$ is 2-connected.

(iii) For any vertex $u\in V(G')$, if $u=v$, then $G'-u=G$ is connected;
if $u\neq v$, then every component of $G-u$ has a vertex adjacent to $v$
in $G'$. In any case, $G'-u$ is connected. Thus $G'$ is 2-connected.

(iv) Choose $u,v\in V(G')$ arbitrarily. Since $G$ is 2-connected,
there is a cycle $C$ containing $u$ and $v$. If $V(C)\backslash
V(G')\neq\emptyset$, then $C$ passes through $x,y$. Let $P$ be the
segment of $C$ from $x$ to $y$ which is not in $G'$.
Using the edge $xy$ to replace $P$,
we get a cycle of $G'$ containing $u,v$. It follows that $G'$ has a
cycle containing any two vertices, and so $G'$ is 2-connected.

(v) If $G-v$ is separable, then $v$ is contained in a cut $\{u,v\}$
of $G$. Hence each component of $G-\{u,v\}$ has a neighbor of $v$,
which implies $N_G(v)$ is not a clique, a contradiction.
\end{proof}

Now we prove Theorem \ref{Thm:strengthening-ErdosGallai}.

\vspace{0.1cm}
\noindent
{\bf Proof of Theorem \ref{Thm:strengthening-ErdosGallai}.}
We prove the theorem by contradiction.
We choose $k$ to be the minimum integer for which there is a counterexample
to Theorem \ref{Thm:strengthening-ErdosGallai}. Let $G$ be such a counterexample
that:\\
(i) $|G|$ is minimized;\\
(ii) $|E(G)\backslash\{xy\}|$ is minimized, subject to (i);\\
(iii) the degree sequence $\tau(G)=(d_1,d_2,\ldots,d_n)$ is the largest,
subject to (i) and (ii).

We claim that $xy\in E(G)$; if not, let $G':=G+xy$ (in the usual meaning).
Note that if $G$ contains no $(x,y)$-path of length
at least $k$, then $G'$ contains no $(x,y)$-path of length
at least $k$. However, $G'$ satisfies (i)(ii) and $\tau(G')>\tau(G)$,
a contradiction.

In the following, we say that a vertex $v\in V(G)\backslash\{x,y\}$ is \emph{feasible}
in $G$ if $d_G(v)\geq k$; otherwise, it is called \emph{non-feasible}. (The
two vertices $x$ and $y$ are neither feasible nor non-feasible.)
We also say that $G$ has a feasible $(x,y)$-path, if the path
has length at least $k$.

\begin{claim}\label{Claim:2-cut} $\{x,y\}$ is not a cut of $G$.
\end{claim}

\begin{proof}
Suppose that $\{x,y\}$ is a cut of $G$. Let
$\mathcal{H}=\{H_1,H_2,\ldots,H_t\}$ be the set of components of
$G-\{x,y\}$. Let $S_i=V(H_i)$ and $G_i=G[S_i\cup\{x,y\}]$
where $i\in [1,t]$. Thus, $G_i$
is 2-connected for $i\in [1,t]$. If $S_i$
contains at least $\frac{|S_i|+1}{2}$ feasible vertices, then by
choice condition (i), there is an $(x,y)$-path of length at least
$k$ in $G_i$, and also in $G$, a contradiction. Therefore, $S_i$
contains at most $\frac{|S_i|}{2}$ feasible vertices for any $i\in
[1,t]$, and there are at most
$\sum_{i=1}^t\frac{|S_i|}{2}=\frac{|G|-2}{2}$ feasible vertices in
$V(G)$, a contradiction.
\end{proof}

\begin{claim}\label{Claim:kgeq5}
$k\geq 5$.
\end{claim}

\begin{proof}
Since $G$ is 2-connected, $G$ has an $(x,y)$-path of length at
least 2. This solves the case of $k\leq 2$. Suppose $k=3$.
If $G$ has no $(x,y)$-path of length at least 3, then
$G$ is the complete 3-partite graph $K_{1,1,n-2}$ (recall that $xy\in E(G)$);
but no vertex in $V(G)\backslash\{x,y\}$ has degree at least 3,
a contradiction. This shows that $k\neq 3$.

Now assume $k=4$. If there is a vertex, say $u$, nonadjacent to $x$
and $y$, then any $(x,y)$-path passing through $u$ is of length at
least 4 (such a path exists by Menger's theorem). Thus, every vertex is adjacent to
either $x$ or $y$. By Claim \ref{Claim:2-cut}, $G-\{x,y\}$ is
connected. Suppose that $G-\{x,y\}$ contains a cycle, say $C$.
By Menger's theorem, we can choose $P_1,P_2$ as two vertex-disjoint
paths from $C$ to $x,y$, respectively, then $P_1\cup C\cup P_2$ contains an
$(x,y)$-path of length at least 4. If $G-\{x,y\}$ has a path $P$ of
length 3, let $u,v$ be the two end-vertices of $P$. If $ux,vy\in E(G)$
or $uy,vx\in E(G)$, then there is an $(x,y)$-path of
length 5. Thus, assume without loss of generality, that both
$u,v\in N(x)$. Thus $C=P[u,v]vxu$ is a cycle of length 5. Let $P'$ be a
path from $P$ to $y$. Then $P'\cup C$ contains an $(x,y)$-path of
length at least 4. Hence a maximal path of $G-\{x,y\}$ is of length
at most 2. We infer $G-\{x,y\}$ is a star $K_{1,n-3}$. The star contains
no feasible vertex for $n\leq 4$; and contains only one feasible
vertex for $n\geq 5$, a contradiction.
\end{proof}

\begin{claim}\label{Claim:nonfeaclique}
If $u$ is non-feasible, then $N(u)$ is a clique.
\end{claim}

\begin{proof}
Suppose $N(u)$ is not a clique. Since $u$ is non-feasible,
$d(u)\leq k-1\leq n-2$. This means $V(G)\backslash (N[u])\neq \emptyset$.
We divide the proof into two cases.

\medskip\noindent\textbf{Case A.} $G-u$ is 2-connected.

Suppose $d(u)=2$ and set $N(u)=\{v_1,v_2\}$. Suppose to the contrary
that $v_1v_2\notin E(G)$. Let $G'$ be from $G-u$
by adding the edge $v_1v_2$. Any vertex in $V(G)\backslash\{u\}$
has the same degree in $G'$ as that in $G$. Note that $x,y\neq u$,
$G'$ is 2-connected and there at least $\frac{|G'|-1}{2}$ vertices
of degree at least $k$. By the choice of $G$, $G'$ has an $(x,y)$-path $P$ of
length at least $k$. If $v_1v_2\notin E(P)$, then $P$ is contained in $G$;
if $v_1v_2\in E(P)$, then using the path $v_1uv_2$ instead
of $v_1v_2$ in $P$, we obtain a feasible $(x,y)$-path
in $G$, a contradiction. Therefore, $v_1v_2\in E(G)$ and
$N(u)$ is a clique.

Now $d(u)\geq 3$. Suppose that $N(u)$ is not a clique. If
there is a vertex $v\in N(u)$ such that it is
non-feasible or $v\in\{x,y\}$, then $G'=G-uv$ is a counterexample
which satisfies (i) but $|E(G')\backslash \{xy\}|<|E(G)\backslash \{xy\}|$,
a contradiction to (ii). Thus, all vertices in $N(u)$ are
feasible. Since $N(u)$ is not a clique, there are two vertices $v,w\in N(u)$
such that $vw\notin E(G)$, and this implies
$N[u]\nsubseteq N[v]$. We have $N[v]\nsubseteq N[u]$ as well,
since $v$ is feasible, which implies $d(v)>d(u)$. Let $G'=G[u\rightarrow v]$.
If $d_{G'}(u)\geq 2$ then let
$G''=G'$; if $d_{G'}(u)=1$ (that is, $u,v$ have no common neighbors
in $G$), then let $G''=G'-u$.
Since $G-u$ is 2-connected and $G'-u$ is a spanning supergraph of $G-u$,
we have $G'-u$ is 2-connected. If $d_{G'}(u)=1$, then $G''=G'-u$ is 2-connected,
if $d_{G'}(u)\geq 2$, then $G''$ is obtained from $G'-u$ by adding the vertex
$u$ and at least two edges incident to $u$, and so $G''$ is 2-connected as well.

Furthermore, we have either $|G''|<|G|$, or $\tau(G'')>\tau(G)$
by Lemma \ref{Lemma:Kelmans}(i). Note that for any vertex
$w\in V(G)\backslash\{u\}$ (including $v$), we have
$d_{G''}(w)\geq d_G(w)$. Since $u$ is non-feasible, there
are also at least $\frac{|G''|-1}{2}$ vertices in
$V(G'')\backslash \{x,y\}$ of degree at least $k$.
By the choice of $G$, $G''$ (and so $G'$) has an $(x,y)$-path of
length at least $k$. By Lemma \ref{Lemma:Kelmans}(ii), $G$
has a feasible $(x,y)$-path, a contradiction.

\medskip\noindent\textbf{Case B.} $G-u$ is separable.

By Lemma \ref{Lem:2connected2}(i), any end-block of $G-u$
has at least one inner-vertex adjacent to $u$ in $G$.

\medskip\noindent\textbf{Subcase B.1.} $G-u$ has an end-block $B$ such that:
$|B|\geq 3$ or $V(B)\nsubseteq N(u)$.

First assume $V(B)\nsubseteq N(u)$. Then there exists an
inner-vertex $v\in V(B)$, such that $uv\in E(G)$ and
$N_B(v)\backslash N(u)\neq \emptyset$, and this implies
$N[v]\nsubseteq N[u]$. Since $u$ has some neighbors outside $B$
(for example, the inner-vertex of another end-block of $G-u$),
$N[u]\nsubseteq N[v]$. Let $G'=G[u\rightarrow v]$.
If $d_{G'}(u)\geq 2$, let $G''=G'$;
if $d_{G'}(u)=1$, let $G''=G'-u$.

Note that $G'-u$ has a spanning subgraph, that is
obtained from $G-u$, by adding an edge between $v$ and
an inner-vertex of each end-block of $G-u$ other than $B$.
By Lemma \ref{Lem:2connected2}(ii), $G'-u$ is 2-connected.
If $d_{G'}(u)=1$, then $G''=G'-u$ is 2-connected; if
$d_{G'}(u)\geq 2$, then $G''=G'$ is obtained from $G'-u$
by adding a vertex $u$ of degree at least 2 in $G'$,
and so $G''$ is 2-connected. In this case, either $|G''|<|G|$,
or $\tau(G'')>\tau(G)$ by Lemma \ref{Lemma:Kelmans}.
Observe that for any vertex $u'\in V(G)\backslash\{u\}$, we have
$d_{G''}(u')\geq d_{G}(u')$. It follows that there are at least
$\frac{|G''|-1}{2}$ vertices of degree at least $k$ in $G''$.
By the choice of $G$, $G''$ has an $(x,y)$-path of length
at least $k$. By Lemma \ref{Lemma:Kelmans}, $G$ has a
feasible $(x,y)$-path, a contradiction.

Now suppose $V(B)\subseteq N(u)$ and $|B|\geq 3$. Let $v$ be an
inner-vertex of $B$. Since $d(v)<d(u)$, $v$ is non-feasible or
$v\in \{x,y\}$. Let $G'=G-uv$. Then $G'$ is obtained
from $G-u$ by adding a new vertex $u$. Note that $u$ is
adjacent to at least one inner-vertex of each end-block
of $G-u$ in $G'$ (here we use the fact that $G$ is 2-connected and $|B|\geq 3$).
By Lemma \ref{Lem:2connected2} (iii), $G'$ is 2-connected.
Now $G'$ is a counterexample which satisfies (i) but
$|E(G')\backslash \{xy\}|<|E(G)\backslash \{xy\}|$, a
contradiction.

\medskip\noindent\textbf{Subcase B.2.} For any end-block $B$ of $G-u$,
$|B|=2$ and $V(B)\subseteq N(u)$.

Let $v$ be a cut-vertex of $G-u$ such that $G-\{u,v\}$ has only one
non-trivial component, i.e., a component with at least two vertices.
(Such a vertex exists since every end-block has two vertices.)
Therefore, $v$ is contained in at least one end-block of $G-u$.
Let $\mathcal{B}=\{B_1,B_2,\ldots,B_t\}$ be the set of the end-blocks
of $G-u$ containing $v$, say, with $V(B_i)=\{v,v_i\}$ for all
$i\in [1,t]$.

Observe that $G-uv$ is 2-connected (recall that
$v$ is a cut-vertex of $G-u$). We shall show
that $v$ is feasible and each $v_i$ is non-feasible.
If $v$ is non-feasible or $v\in \{x,y\}$, then
$G-uv$ is a counterexample with
$|E(G-uv)\backslash \{xy\}|<|E(G)\backslash \{xy\}|$,
a contradiction. Since $u,v\notin \{x,y\}$ and $N_G(v_i)=\{u,v\}$,
by the fact $xy\in E(G)$, we have $v_i\notin \{x,y\}$.
Since $d(v_i)=2$, $v_i$ is non-feasible by Claim \ref{Claim:kgeq5}.

Clearly $G-u$ is not a star. Since $u$ has some neighbor which is
an inner-vertex of an end-block of $G-u$ not in $\mathcal{B}$, we
have $N[u]\nsubseteq N[v]$. Since $v$ is feasible, we infer
$N[v]\nsubseteq N[u]$ as well. Let $G'=G[u\rightarrow v]$.
If $u$ and $v$ have a common neighbor other than $v_i$,
$i\in[1,t]$, then let $G''=G'$; otherwise ($v$ is
a cut-vertex of $G'$), let $G''=G'-\{u,v_1,\ldots,v_t\}$.

Next, we aim to show that $G''$ is 2-connected. We
first show that $G'-\{u,v_1,...,v_t\}$
is 2-connected. For any end-block of $G-u$, if it is not contained
in $\mathcal{B}$, then it has an inner vertex that is adjacent to
$v$ in $G'-\{u,v_1,...,v_t\}$. By Lemma \ref{Lem:2connected2}(ii),
$G'-\{u,v_1,...,v_t\}$ is 2-connected. If $v$
is a cut-vertex of $G'$, then $G''=G'-\{u,v_1,\ldots,v_t\}$ is 2-connected;
otherwise, $G''=G'$ is obtained from the 2-connected graph $G'-\{u,v_1,\ldots,v_t\}$,
by adding vertices $u,v_1,\ldots,v_t$, such that $u$ is adjacent
to at least two vertices in $V(G')\backslash \{u,v_1,...,v_t\}$
and $v_i$ is adjacent to $\{u,v\}$ for $i=1,\ldots,t$. Thus, $G''$ is 2-connected.

By Lemma \ref{Lemma:Kelmans}, either $|G''|<|G|$ or $\tau(G'')>\tau(G)$.
Note that every vertex in
$V(G)\backslash\{u,v,v_1,\ldots,v_t\}$ has the same degree in $G''$
as that in $G$. If $G''=G'$, then $v$ has degree in $G'$ greater
than that in $G$; if $G''\neq G'$, then $|G''|\leq|G|-2$. In both
cases, there are at least $\frac{|G''|-1}{2}$ vertices in
$V(G'')\backslash\{x,y\}$ with degree at least $k$. It follows that
$G''$ (and then $G'$) has an $(x,y)$-path of length at
least $k$. By Lemma \ref{Lemma:Kelmans}, there is a feasible $(x,y)$-path in
$G$, a contradiction.
\end{proof}

\begin{claim}\label{Claim:xyneighbor}
Both $x$ and $y$ have at least two neighbors in $V(G)\backslash\{x,y\}$.
\end{claim}

\begin{proof}
We prove the claim for $x$.
Suppose $N(x)=\{x',y\}$. If $x'y\in E(G)$, let $G'=G-x$; otherwise,
let $G'$ be the graph obtained from $G-x$ by adding a new edge $x'y$.
Thus $G'$ is 2-connected by Lemma \ref{Lem:2connected2} (iv). If every
vertex in $V(G')\backslash\{x',y\}$ is feasible, then it has degree
at least $k$ in $G'$ as well, and by Erd\H{o}s-Gallai Theorem \cite{EG59},
$G'$ has an $(x',y)$-path $P$ of length at least $k$.
Thus, $P'=xx'P$ is a feasible $(x,y)$-path in $G$, a contradiction.
Thus, there is a non-feasible vertex $u$ in $V(G')\backslash\{x',y\}$.
Let $G''=G'-u$. Recall that $N_G(u)$ is a clique, implying that
$G''$ is 2-connected by Lemma \ref{Lem:2connected2} (v).
Note that every feasible vertex in $V(G')\backslash\{x',y\}$
has degree at least $k-1$ in $G''$. There are at least
$\frac{|G|-1}{2}-1=\frac{|G''|-1}{2}$ such vertices.
It follows that $G''$ has an $(x',y)$-path $P$
of length at least $k-1$. Thus, $P'=xx'P$ is a feasible
$(x,y)$-path in $G$, a contradiction. By symmetry, we
can prove the statement for $y$.
\end{proof}

\begin{claim}\label{Claim:G-x2con}
Both $G-x$ and $G-y$ are 2-connected.
\end{claim}

\begin{proof}
We prove the claim for $x$. Suppose that $G-x$ is separable.
By Claim \ref{Claim:2-cut}, $\{x,y\}$ is not a cut of $G$.
This implies $y$ is not a cut-vertex of $G-x$.
Let $B$ be the unique block of $G-x$ which contains
$y$. Let $Y=\{y_1,y_2,\ldots,y_t\}$ be the set of
cut-vertices of $G-x$ contained in $B$. We remark
that possibly $|Y|=1$. Since the neighborhood of $y_i$ is
not a clique, by Claim \ref{Claim:nonfeaclique}, $y_i$ is feasible.
And $y_i$ is possibly contained in more than two
blocks of $G-x$. Let $S_i$ be the set of vertices of the
components of $G-\{x,y_i\}$ not containing $y$, and
$G_i=G[S_i\cup\{x,y_i\}]$. If $xy_i\notin E(G)$,
then we add the edge $xy_i$ to $G_i$. By Lemma \ref{Lem:2connected2} (iv),
$G_i$ is 2-connected. Recall $y$ is not a cut-vertex 
of $G-x$ and $d_{G-x}(y)\geq 2$,
and so $|B|\geq 3$. This implies that $B$ has a
$(y_i,y)$-path of length at least 2.

\begin{fact}
If $|S_i|\geq 2$, then $S_i$ contains at most $\frac{|S_i|-2}{2}$
feasible vertices.
\end{fact}

\begin{proof}
Suppose that $S_i$ contains at least $\frac{|S_i|-1}{2}$ feasible
vertices. If $S_i$ contains 0, 1, or at least 2 non-feasible
vertices, then let $T_i\subseteq S_i$ be a set of 0, 1, or exactly
2 non-feasible vertices, respectively. Let $G'_i=G_i-T_i$.
Since the neighborhood of every non-feasible vertex is a clique,
$G'_i$ is 2-connected by Lemma \ref{Lem:2connected2} (v). Note
that every feasible vertex in $S_i$ has degree at least $k-2$ in
$G'_i$. If $|T_i|\leq 1$, then all vertices in $S_i\backslash T_i$
have degree at least $k-2$; if $|T_i|=2$, then there are at least
$\frac{|S_i|-1}{2}=\frac{|G'_i|-1}{2}$ vertices in $S_i\backslash
T_i$ of degree at least $k-2$. For each case, $G'_i$ has an
$(x,y_i)$-path $P$ of length at least $k-2$. Let $Q$ be a
$(y_i,y)$-path in $B$ of length at least 2. Then $P'=Py_iQ$
is a feasible $(x,y)$-path in $G$, a contradiction.
\end{proof}

Let $G'$ be the graph obtained from $B$ by adding a new vertex
$x'$, an edge $x'y$, and all $t$ edges $x'y_i$, $i\in[1,t]$.
Since $B$ is 2-connected and $d_{G'}(x')\geq 2$, $G'$ is 2-connected.
A feasible vertex in $V(G)\backslash(\{x,y\}\cup
\{y_i\in Y: |S_i|\geq 2\}\cup\bigcup_{i=1}^tS_i)$ has degree at
least $k-1$ in $G'$. If $|S_i|=1$, then clearly the
vertex in $S_i$ is non-feasible.
Therefore, in $V(G')\backslash\{x',y\}$, there are at
least
$$\frac{|G|-1}{2}-\sum_{|S_i|\geq 2}\left(\frac{|S_i|-2}{2}+1\right)\geq\frac{|G'|-1}{2}$$
vertices of degree at least $k-1$ in $G'$. Notice that
$|G'|<|G|$. Hence $G'$ has an $(x',y)$-path $P$ of length at least
$k-1$. Let $x'y_i$ be the first edge on $P$,
and $Q$ be an $(x,y_i)$-path of $G_i$ of length at least 2.
Then $P'=Qy_iP[y_i,y]$ is a feasible $(x,y)$-path in
$G$, a contradiction. The other
assertion can be proved by symmetry.
This proves Claim \ref{Claim:G-x2con}.
\end{proof}

\begin{claim}\label{Claim:feasiblenumber}
There are exactly $\frac{|G|-1}{2}$ feasible vertices in $V(G)\backslash\{x,y\}$,
and every vertex in $N(x)\backslash\{y\}$ ($N(y)\backslash\{x\}$) is feasible.
\end{claim}

\begin{proof}
Suppose to the contrary. If there is a non-feasible vertex in
$N(x)\backslash\{y\}$, then let $x'$ be such a
vertex; if all vertices in $N(x)\backslash\{y\}$ are feasible and
there are at least $\frac{|G|}{2}$ feasible vertices in
$V(G)\backslash\{x,y\}$, then choose $x'\in N(x)\backslash\{y\}$
arbitrarily. Let $G'=G-x$. By Claim \ref{Claim:G-x2con},
$G'$ is 2-connected. Note that every feasible vertex in
$V(G')\backslash\{x',y\}$ has degree at least $k-1$ in $G'$. There
are at least
$$\min\left\{\frac{|G|-1}{2},\frac{|G|}{2}-1\right\}=\frac{|G'|-1}{2}$$
such vertices. It follows $G'$ has an $(x',y)$-path $P$ of length
at least $k-1$. Then $P'=xx'P$ is a feasible $(x,y)$-path
in $G$, a contradiction.
If there is a vertex in $N(y)\backslash\{x\}$ that is non-feasible,
then we can prove it similarly.
\end{proof}

\begin{claim}\label{Claim:uxcommon}
Let $u$ be a non-feasible vertex. Then $u$ and $x$ have at least two common
neighbors; so do $u$ and $y$.
\end{claim}

\begin{proof}
Suppose not. If $u$ and $x$ have exactly one common neighbor, then
let $x'\in N(x)\cap N(u)$; if $u$ and $x$ have no common
neighbor, then choose $x'\in N(x)\backslash\{y\}$ arbitrarily.
By Claim \ref{Claim:feasiblenumber}, $uy\notin E(G)$, and so $x'\neq y$.
Again, by Claim \ref{Claim:feasiblenumber}, $x'$ is feasible.

By Claim \ref{Claim:G-x2con}, $G-x$ is 2-connected. By
Claim \ref{Claim:nonfeaclique}, the neighborhood of $u$
is clique in $G$. Recall that $u\notin N(x)$. Let
$G':=G-\{x,u\}$. We infer $G'$ is 2-connected by
Lemma \ref{Lem:2connected2}(v). Since $u$ and $x$
have at most one common neighbor, every feasible
vertex other than $x'$ is adjacent to at most one
vertex of $\{u,x\}$. It follows that every feasible
vertex other than $x'$ has degree at least $k-1$ in
$G'$, and hence there are at least
$\frac{|G|-1}{2}-1=\frac{|G'|-1}{2}$ such vertices
in $V(G')\backslash \{x',y\}$. Therefore, $G'$ has an
$(x',y)$-path $P$ of length at least $k-1$. Then
$P'=xx'P$ is a feasible $(x,y)$-path in $G$,
a contradiction. The other assertion can be proved
by symmetry.
\end{proof}

By Claim \ref{Claim:feasiblenumber}, there exist non-feasible
vertices. Moreover, Claims \ref{Claim:feasiblenumber} and \ref{Claim:uxcommon} tell
us all non-feasible vertices in $G$ are at distance 2 from $x$ and $y$.

In the following, let $x'\in N(x)$ such that it has a non-feasible neighbor.

\begin{claim}\label{Claim:G-xx'separable}
$G-\{x,x'\}$ is separable, and $y$ is not a cut-vertex of
$G-\{x,x'\}$.
\end{claim}

\begin{proof}
Let $G':=G-\{x,x'\}$ and $x''\in N(x')$ be non-feasible. Suppose
$G'$ is 2-connected. Note that every feasible vertex
has degree at least $k-2$ in $G'$. There are at least
$\frac{|G|-1}{2}-1=\frac{|G'|-1}{2}$ such vertices
in $G'\backslash \{x'',y\}$. Thus, $G'$ has an $(x'',y)$-path
$P$ of length at least $k-2$. Then $P'=xx'x''P$
is a feasible $(x,y)$-path in $G$, a contradiction.
Thus, $G'$ is separable.

Suppose $y$ is a cut-vertex of $G-\{x,x'\}$. Let
$\mathcal{H}=\{H_1,H_2,\ldots,H_t\}$ be the set of components of
$G-\{x,x',y\}$, $S_i=V(H_i)$, and $G_i=G[S_i\cup\{x',y\}]$, $1\leq i\leq t$. If
$x'y\notin E(G)$, then we add an edge $x'y$ to each $G_i$.
Note that $G-x$ is 2-connected, and so $G_i$
is 2-connected by Lemma \ref{Lem:2connected2} (iv).

If $|S_i|=1$ then the vertex in $S_i$ is non-feasible (it has degree
at most 3) and adjacent to $y$, contradicting Claim
\ref{Claim:feasiblenumber}. Hence $|S_i|\geq 2$ for
every $i\in[1,t]$.

\begin{fact}\label{Fact:2}
$S_i$ contains at most $\frac{|S_i|-1}{2}$ feasible vertices.
\end{fact}

\begin{proof}
Note that every feasible vertex in $S_i$ has degree at least $k-1$
in $G_i$. If every vertex in $S_i$ is feasible,
then $G_i$ has an $(x',y)$-path $P$ of length
at least $k-1$. Therefore, $P'=xx'P$ will be a feasible
$(x,y)$-path in $G$, a contradiction. Thus, $S_i$ contains
a non-feasible vertex, say $v_i$, for every $i\in [1,t]$.
By Claim \ref{Claim:uxcommon}, $x$ and $v_i$ have at least
two common neighbors. Since $v_i$ is non-feasible and
$N(y)\backslash \{x\}$ consists of feasible vertices,
$v_i\notin N(y)$, and it follows $x$ has a neighbor in $S_i$ for every
$i\in[1,t]$. Notice that $x'$ also has a neighbor in $S_i$.
Thus, there is an $(x,x')$-path of length at least 2 with all
internal vertices in $S_i$ for any $i\in [1,t]$.

Suppose $S_i$ contains at least $\frac{|S_i|}{2}$ feasible
vertices. Recall that $v_i\in S_i$ is a non-feasible vertex. Let
$G'_i=G_i-v_i$. Then $G'_i$ is 2-connected by Lemma \ref{Lem:2connected2}
(v) (note that $G_i$ is 2-connected and $N_{G_i}(v_i)$ is a
clique). A feasible vertex in $S_i$ has degree at least $k-2$
in $G'_i$. There are at least $\frac{|S_i|}{2}=\frac{|G'_i|-1}{2}$
vertices in $V(G'_i)\backslash\{x',y\}$ of degree at least $k-2$ in
$G'_i$. Thus, $G'_i$ has an $(x',y)$-path $P$ of length at least $k-2$.
Let $Q$ be an $(x,x')$-path of length at least 2 with all internal vertices in
$S_j$ with $j\neq i$ (by the analysis above, such a path exists).
Then $P'=Qx'P$ is a feasible $(x,y)$-path, a contradiction.
\end{proof}

By Fact \ref{Fact:2}, there are at most
$|\{x'\}|+\sum_{i=1}^t\frac{|S_i|-1}{2}=\frac{|G|-t-1}{2}\leq\frac{|G|-2}{2}$
feasible vertices, a contradiction. This proves Claim
\ref{Claim:G-xx'separable}.
\end{proof}

Recall $x'\in N(x)$ such that it has a non-feasible neighbor.
By Claim \ref{Claim:G-xx'separable}, $y$ is contained
in a unique block of $G-\{x,x'\}$, say
$B$.

To finish the proof, we only need to consider
the cases according to whether $B$ is an end-block of $G-\{x,x'\}$
or not. The coming claim solves the first case.
\begin{claim}\label{Claim:BnotEnd}
$B$ is not an end-block of $G-\{x,x'\}$.
\end{claim}

\begin{proof}
Suppose $B$ is an end-block of $G-\{x,x'\}$. Let $y'$ be the
cut-vertex of $G-\{x,x'\}$ contained in $B$. Let
$\mathcal{H}=\{H_1,H_2,\ldots,H_t\}$ be the set of components of
$G-\{x,x',y'\}$ not containing $y$, $S'=V(B)\backslash\{y',y\}$,
and $S_i=V(H_i)$ for $i\in [1,t]$.
(Hence $V(G)=\{x,x',y',y\}\cup S'\cup\bigcup_{i=1}^tS_i$.) We remark
that possibly $S'=\emptyset$.

\begin{fact}\label{Claim:S_i'feasible}
$S_i$ contains at most $\frac{|S_i|-1}{2}$ feasible vertices for
$i\in [1,t]$.
\end{fact}

\begin{proof}
Suppose $S_i$ contains at least $\frac{|S_i|}{2}$ feasible
vertices. Let $G_i:=G[S_i\cup \{x',y'\}]$. If $x'y'\notin E(G)$,
then add the new edge $x'y'$ to $G_i$. Recall that $G-x$
is 2-connected. By Lemma \ref{Lem:2connected2} (iv),
$G_i$ is 2-connected.

Let $T_i\subseteq S_i$ be a set of 0, or 1 non-feasible
vertex, if $S_i$ contains 0 or at least 1 non-feasible vertex,
respectively. Let $G'_i=G_i-T_i$. By Lemma \ref{Lem:2connected2}
(v), $G'_i$ is 2-connected. Note
that every feasible vertex in $S_i$ has degree at least $k-2$ in
$G'_i$. If $T_i=\emptyset$, then every vertex in $S_i$ has degree at
least $k-2$ in $G'_i$; if $|T_i|=1$, then at least
$\frac{|S_i|}{2}=\frac{|G'_i|-1}{2}$ vertices in $S_i\backslash T_i$
has degree at least $k-2$ in $G'_i$. For any case, $G'_i$ has an
$(x',y')$-path, say $P$, of length at least $k-2$. Let $Q$ be a
$(y',y)$-path in $B$. Then $P'=xx'Py'Q$ is a feasible
$(x,y)$-path in $G$, a contradiction.
\end{proof}

\begin{fact}\label{Claim:S'feasible}
$S'$ contains at most $\frac{|S'|}{2}$ feasible vertices.
\end{fact}

\begin{proof}
If $S'=\emptyset$, then the assertion is trivial. Suppose now that
$S'\neq\emptyset$ and $S'$ contains at least $\frac{|S'|+1}{2}$
feasible vertices. Let $G'=G[S'\cup\{y',y\}]$. Note that every
feasible vertex in $S'$ has degree at least $k-2$ in $G'$. There are
at least $\frac{|S'|+1}{2}=\frac{|G'|-1}{2}$ such vertices. It follows that $G'$
has a $(y',y)$-path $P$ of length at least $k-2$. Let $Q$ be an
$(x',y')$-path with all internal vertices in $H_1$. Then
$P'=xx'Qy'P$ is a feasible $(x,y)$-path in $G$,
a contradiction.
\end{proof}

Now by the above two facts, $G$ has at most
$$2+\sum_{i=1}^t\frac{|S_i|-1}{2}+\frac{|S'|}{2}=\frac{|G|-t}{2}$$
feasible vertices, implying that $t=1$.

Therefore, $G-\{x,x'\}$ consists of $B$ and $B_1:=G[S_1\cup \{y'\}]$. Next,
we claim that $B_1$ is nonseparable. Suppose that $B'_1$ is an end-block
of $B_1$ not containing $y'$. If all inner vertices of $B'_1$ are feasible,
then by Erd\H{o}s-Gallai Theorem, there is an $(x',y')$-path $P$ of length
at least $k-1$ with all internal vertices in $S_1$. Thus, $G$ has a
feasible $(x,y)$-path, and so $B'_1$ has an inner vertex which is
non-feasible, say $v$. Recall Claim \ref{Claim:uxcommon}, $vx\notin E(G)$.
Thus, $N(y)\cap N(v)\subseteq \{y',x'\}$. By Claim \ref{Claim:uxcommon},
we have  $vy'\in E(G)$, a contradiction. We conclude that $G-\{x,x'\}$
consists of two blocks $B$ and $B_1$ with a common vertex $y'$.

By Fact \ref{Claim:S_i'feasible}, $S_1$ contains a non-feasible
vertex, say $x''$. By Claim \ref{Claim:uxcommon}, $x'x'',x'y\in
E(G)$. Let $G'$ be the graph obtained from $G-\{x,x'\}$ by adding
the edge $x''y$. Then $G'$ is 2-connected. Note that every feasible
vertex in $V(G')\backslash\{x'',y\}$ has degree at least $k-2$ in
$G'$. There are at least $\frac{|G|-1}{2}-1=\frac{|G'|-1}{2}$ such
vertices in $V(G')\backslash\{x'',y\}$. Hence $G'$ has an
$(x'',y)$-path $P$ of length at least $k-2$, and $xx'x''P$
is a feasible $(x,y)$-path in $G$, a contradiction. This
proves Claim \ref{Claim:BnotEnd}.
\end{proof}
By Claim \ref{Claim:BnotEnd}, $B$ is not an end-block of $G-\{x,x'\}$.
Let $Y=\{y_1,y_2,\ldots,y_t\}$ be the cut vertices of $G-\{x,x'\}$
contained in $B$. By Claim \ref{Claim:BnotEnd}, $|Y|\geq 2$. Let
$S_i$ be the set of the vertices of the components of
$G-\{x,x',y_i\}$ not containing $y$, and $G_i=G[S_i\cup\{x',y_i\}]$.
If $x'y_i\notin E(G)$, then we add the edge $x'y_i$ to $G_i$.
Thus, $G_i$ is 2-connected for $i\in[1,t]$.

For any $i\in [1,t]$, every feasible vertex in $S_i$ has degree at least $k-1$
in $G_i$. If every vertex in $S_i$ is feasible, then $G_i$ has an
$(x',y_i)$-path $P$ of length at least $k-1$. Let $Q$ be a
$(y_i,y)$-path in $B$. Then $P'=xx'Py_iQ$ is a
feasible $(x,y)$-path in $G$, a contradiction. Thus, every
$S_i$ contains a non-feasible vertex. By
Claim \ref{Claim:uxcommon}, $x'y\in E(G)$ and $yy_i\in E(G)$
for every $i\in[1,t]$.

We claim that there is an $(x,x')$-path of length at least 3 with
all internal vertices in $S_i\cup\{y_i\}$ for $i\in [1,t]$. If there are two
disjoint edges from $x$ and $x'$, respectively, to
$S_i\cup\{y_i\}$, then the assertion is trivial. Now assume that $v$ is
the only neighbor of $x$ and $x'$ in $S_i\cup\{y_i\}$. By Claim
\ref{Claim:feasiblenumber}, $v$ is feasible. Let $v'$ be a
non-feasible vertex contained in $S_i$. Then $v$ is the only
possible common neighbor of $x$ and $v'$, contradicting Claim
\ref{Claim:uxcommon}. Thus as we claimed, there is an $(x,x')$-path
of length at least 3 with all internal vertices in $S_i\cup\{y_i\}$,
for all $i\in[1,t]$.

\begin{claim}\label{Claim:Sineq2}
$|S_i|\neq 2$ for all $i=1,\ldots,t$.
\end{claim}

\begin{proof}
Without loss of generality, suppose that $|S_1|=2$. Each vertex
in $S_1$ has degree at most 4 in $G$, and so is non-feasible by
Claim \ref{Claim:kgeq5}. By Claims \ref{Claim:feasiblenumber}
and \ref{Claim:uxcommon},
$xy_1\in E(G)$. Recall that we have proved $yy_i\in E(G)$ for every
$i\in[1,t]$.

Let $G'$ be the graph obtained from $G-(\{x,x'\}\cup S_1)$, by
adding all edges $y_1v$ for any vertex $v\in N(x')\cap \bigcup_{i=2}^tS_i$.
Since $G-x$ is 2-connected and $G-\{x,x'\}$ is separable,
$x'$ is adjacent to an inner-vertex of each end-block of
$G-\{x,x'\}$. Thus, every end-block of $G-(\{x,x'\}\cup S_1)$
has an inner-vertex which is adjacent to $y_1$ in $G'$.
By Lemma \ref{Lem:2connected2} (ii), $G'$ is 2-connected.

Observe that every feasible vertex in $V(G')\backslash\{y_1,y\}$ has
degree at least $k-2$ in $G'$, and there are at least
$\frac{|G|-1}{2}-2=\frac{|G'|-1}{2}$ such vertices. It follows that
$G'$ has a $(y_1,y)$-path $P$ of length at least $k-2$. Let $y_1v$
be the first edge on $P$.

Let $Q$ be an $(x',y_1)$-path with all internal vertices in $S_1$.
If $v\in V(B)$, i.e., $y_1v$ is not an edge in $E(G')\backslash
E(G)$, then $P'=xx'Qy_1P$ is a feasible $(x,y)$-path
in $G$, a contradiction. If $v\in S_i$ for some $i\in [2,t]$, i.e.,
the edge $y_1v$ is in $E(G')\backslash E(G)$, then $x'v\in E(G)$
and $P'=xy_1Qx'vP[v,y]$ is a feasible $(x,y)$-path in
$G$, also a contradiction. This proves Claim \ref{Claim:Sineq2}.
\end{proof}

Note that if $|S_i|=1$, then the vertex in $S_i$ has degree at most
3, an so it is non-feasible. To continue the proof, we need to analyze
the case of $S_i\geq 3$.

\begin{claim}\label{Claim:Sifeasible}
If $|S_i|\geq 3$, then $S_i$ contains at most $\frac{|S_i|-3}{2}$
feasible vertices.
\end{claim}

\begin{proof}
Suppose that $S_i$ contains at least $\frac{|S_i|-2}{2}$ feasible
vertices. Let $T_i\subseteq S_i$ be a set of 1, 2 or 3 non-feasible
vertices, if $S_i$ contains 1, 2 or at least 3 non-feasible
vertices, respectively. Let $G'_i=G_i-T_i$. We note that every feasible
vertex in $S_i$ has degree at least $k-4$ in $G'_i$. If $|T_i|\leq
2$, then all vertices in $S_i\backslash T_i$ has degree at least
$k-4$; if $|T_i|=3$, then at least
$\frac{|S_i|-2}{2}=\frac{|G'_i|-1}{2}$ vertices in $S_i\backslash
T_i$ has degree at least $k-4$. For any case, $G'_i$ has
at least 3 vertices and is 2-connected. Thus, it has an
$(x',y_i)$-path $P$ of length at least $k-4$.

Let $j\in [1,t]$ with $j\neq i$. Recall that there is an
$(x,x')$-path, say $Q$, of length at least 3 with all
internal vertices in $S_j\cup\{y_j\}$ (see the analysis
before Claim \ref{Claim:Sineq2}). Since $y_iy\in E(G)$,
$P'=Qx'Py_iy$ is a feasible $(x,y)$-path
in $G$, a contradiction.
\end{proof}

Finally we are completing the proof using the above two claims.
Let $G'$ be the graph obtained from
$G-(\{x,x'\}\cup\bigcup_{i=1}^tS_i)$ by adding a new vertex $x''$,
an edge $x''y$ and $t$ edges $x''y_i$, $i\in [1,t]$. Each feasible
vertex in $V(G)\backslash(\{x,x',y\}\cup\{y_i: |S_i|\geq
3\}\cup\bigcup_{i=1}^tS_i)$ has degree at least $k-2$ in $G'$. By
Claims \ref{Claim:Sineq2} and \ref{Claim:Sifeasible}, $G'$ has at
least
$$\frac{|G|-1}{2}-\sum_{|S_i|\geq
3}\left(\frac{|S_i|-3}{2}+1\right)-1\geq\frac{|G|-1}{2}-\sum_{i=1}^t\frac{|S_i|-1}{2}-1\geq\frac{|G'|-1}{2}$$
vertices in $V(G')\backslash \{x'',y\}$ of degree at least
$k-2$ in $G'$. Thus $G'$ has an $(x'',y)$-path $P$ of length
at least $k-2$. We note that $x''y_i\in E(P)$ for some $i\in [1,t]$.
Set $Q$ be an $(x',y_i)$-path of length at least
2 in $G_i$. Then $P'=xx'Qy_iP[y_i,y]$
is a feasible $(x,y)$-path in $G$. The proof of
Theorem \ref{Thm:strengthening-ErdosGallai} is complete.
\hfill$\Box$

\section{Concluding remarks}\label{Sec:Remarks}
In this paper, we focus on a conjecture of Woodall on
cycles and improvements of Erd\H{o}s-Gallai Theorem
on paths. We also use these as tools to give short proofs
of known theorems and make progress on a conjecture of Bermond.
In what follows, we discuss related problems,
some of which shall motivate our future research.

In 1985, H\"{a}ggkvist and Jackson \cite{HJ85} suggested
a strengthening of Woodall's conjecture.
\begin{conj}[H\"{a}ggkvist, Jackson \cite{HJ85}]\label{Conj:HJ}
Let $G$ be a 2-connected graph on $n$ vertices. If $G$ contains
at least $\max\{2k-1,\frac{n+k}{2}+1\}$
vertices of degree at least $k$,
then $G$ has a cycle of length at least $\min\{n,2k\}$.
\end{conj}

H\"{a}ggkvist and Jackson \cite{HJ85} constructed the following two
classes of graphs. Let $G_1:=K_2\vee (K_{2k-4}+\overline{K_t})$.
Let $H_1:=K_{\frac{k-1}{2}}\vee \overline{K_{\frac{k-1}{2}}}$
where $k\geq 5$ is odd and $H_2=K_{k+1}$. Let $G_2$ be the
graph obtained from one copy of $H_2$ and several disjoint
copies of $H_1$ by joining each vertex in the
$K_{\frac{k-1}{2}}$ subgraph of $H_1$ to two fixed vertices of
$H_2$. One can see $G_1$ has $2k-2$ vertices of degree at least $k$
and $c(G_1)=2k-1$; $G_2$ has $\frac{n+k+1}{2}$ vertices
of degree at least $k$ and $c(G_2)=2k-1$. Thus,
Conjecture \ref{Conj:HJ}, if true, will be sharp by
these examples.

In 2013, Li \cite[Conjecture~4.14]{L13} conjectured
that for any 2-connected graph $G$ of order $n$, there is
a cycle of length at least $2k$ if the number of vertices
of degree at least $k$ is at least $\frac{n+k}{2}$. The
constructions $G_1$ and $G_2$ mentioned above disprove
Li's conjecture.

In closing, we suggest the following conjecture,
which is a generalization of Theorem \ref{Thm:strengthening-ErdosGallai}
(set $\alpha=\frac{1}{2}$).
\begin{conj}\label{Conj:LN}
Let $G$ be a 2-connected graph on $n$ vertices and
$x,y\in V(G)$. Let $0<\alpha\leq \frac{1}{2}$.
If $G-\{x,y\}$ contains more than $\alpha(n-2)$
vertices of degree at least $k$, then $G$ contains an
$(x,y)$-path of length at least $2\alpha k$.
\end{conj}
For any rational number $\alpha$, we choose $k$ such that
$\alpha(k-1)$ is an integer at least 2. Let
$H=K_{\alpha(k-1)}\vee \overline{K_{(1-\alpha)(k-1)}}$.
Let $G$ be obtained from $t$ copies of $H$, by adding two
new vertices $\{x,y\}$ and all possible edges between $\{x,y\}$ and
the $K_{\alpha(k-1)}$ subgraph of each $H$. The number of vertices of
degree at least $k$ in $G$ is $\alpha\cdot (|G|-2)$ and a longest
$(x,y)$-path is of length $2\alpha(k-1)$. This example
shows that Conjecture \ref{Conj:LN} is sharp for infinite
values of integers $n$ and $k$.

If Conjecture \ref{Conj:LN} is true, we can make partial
progress on Conjecture \ref{Conj:HJ}.

\bigskip

\noindent {\bf Acknowledgement.} The authors are very grateful
to Douglas Woodall for sending a copy of \cite{W75}
to them. They are also very grateful to Xing Peng
for many helpful comments.

\end{document}